\documentclass[preprint]{imsart}

\RequirePackage[OT1]{fontenc}
\RequirePackage{amsthm,amsmath,amssymb}
\RequirePackage{natbib}
\RequirePackage[colorlinks,citecolor=blue,urlcolor=blue]{hyperref}
\usepackage{mathtools, enumitem}
\mathtoolsset{showonlyrefs} 
\usepackage[margin=27mm]{geometry}
\newcommand{\rd}{\mathrm{d}}
\newcommand{\T}{{ \mathrm{\scriptscriptstyle T} }}

\startlocaldefs
\numberwithin{equation}{section}
\theoremstyle{plain}
\newtheorem{thm}{Theorem}[section]
\newtheorem{lem}{Lemma}[section]
\newtheorem{corollary}{Corollary}[section]
\theoremstyle{remark}
\newtheorem{remark}{Remark}[section]
\theoremstyle{definition}

\endlocaldefs

\begin{document}

\begin{frontmatter}
\title{A Review of Brown 1971 (in)admissibility results under scale mixtures of Gaussian priors}
\runtitle{Admissibility}

\begin{aug}
\author{\fnms{Yuzo} \snm{Maruyama}
\ead[label=e1]{maruyama@port.kobe-u.ac.jp}}
\and
\author{\fnms{William, E.} \snm{Strawderman}
\ead[label=e2]{straw@stat.rutgers.edu}}

\address{Kobe University and Rutgers University \\
\printead{e1,e2}}


\runauthor{Y.~Maruyama and W.~E.~Strawderman}

\affiliation{Kobe University and Rutgers University}

\end{aug}

\begin{abstract}
Brown's 1971 paper 
``Admissible estimators, recurrent diffusions and insoluble boundary value problems'' 
is a landmark in the admissibility literature. 
It nearly completely settles the issue of admissibility/inadmissibility for estimating the mean of a multivariate normal distribution with identity covariance under sum of squared error loss. 
We revisit this wonderful tour de force on its 50th anniversary and present an alternative and more direct proof of the result for generalized Bayes estimators corresponding to priors which are a subclass of scale mixtures of spherical normals.
\end{abstract}


\end{frontmatter}

\section{Introduction}
\label{sec:intro}
Brown's 1971 paper 
``Admissible estimators, recurrent diffusions and insoluble boundary value problems'' is a landmark in the admissibility literature. 
It nearly completely settles the issue of admissibility for estimating the mean of a multivariate normal distribution with identity covariance under sum of squared error loss. 
We revisit this wonderful tour de force on its 50th anniversary and present an alternative and more direct proof of the result for generalized Bayes estimators corresponding to priors which are a subclass of scale mixtures of spherical normals.

Specifically, let
\begin{align*}
 X\sim N_d(\mu,I)
\end{align*}
with  density
\begin{align*}
 \phi(x-\mu)=\frac{1}{(2\pi)^{d/2}}\exp\left(-\frac{\|x-\mu\|^2}{2}\right).
\end{align*}
Consider estimation of $\mu$
under quadratic loss 
$\|\delta-\mu\|^2$.
Let $\Pi(\rd \mu)$ and $m_{\Pi}$ be the prior measure and the corresponding marginal density given by
\begin{align}\label{m_F}
 m_\Pi(x)=\int\phi(x-\mu)\Pi(\rd\mu).
\end{align} 
Then the (generalized) Bayes estimator under $\Pi(\rd \mu)$ is given by 
\begin{align}
 \delta_{\Pi}&=\frac{\int \mu \phi(x-\mu)\Pi(\rd\mu)}{\int \phi(x-\mu)\Pi(\rd\mu)} \\
&=x+\frac{\int (\mu-x) \phi(x-\mu)\Pi(\rd\mu)}{\int \phi(x-\mu)\Pi(\rd\mu)} \\
&=x+ \nabla_x\log m_{\Pi}(x). \label{bayes.esti.F} 
\end{align}
We are interested in determining admissibility/inadmissibility of generalized Bayes estimators, 
for which \cite{Brown-1971} has given an essentially complete solution. 
Among the many results of \cite{Brown-1971}, the following result, Theorem 3.6.1 seems to be the most often quoted.
\begin{quote}
\normalsize\textbf{Theorem} Suppose $\Pi$ is spherically symmetric. Hence $m_{\Pi}(x)=m_{\Pi}( \|x\|^2 )$. If
\begin{align}\label{Brown1971.infty.integral}
 \int_1^\infty \frac{\rd t}{t^{d/2}m_{\Pi}(t)}<\infty,
\end{align}
then $\delta_{\Pi}$ is inadmissible. If the integral is infinite and the risk of $\delta_{\Pi}$ is
bounded or equivalently
\begin{align}\label{Brown1971.infty}
\sup_{\mu}E[\|\nabla_x\log m_{\Pi}(\|X\|^2)\|^2]<\infty,
\end{align}
then $\delta_{\Pi}$ is admissible.
\end{quote}
In \cite{Brown-1971}, however, this statistically important result 
is derived as a corollary of quite deep mathematical results, which many find difficult to follow.
\cite{Brown-Hwang-1982} provide a sufficient condition for admissibility of generalized Bayes estimators, which is
much more readable than \cite{Brown-1971} and is 
based directly on the prior $\pi(\mu)$. 
However, the sufficient condition for admissibility is
a proper subset of Brown's. For example,
some important priors including \cite{Stein-1974}'s prior 
\begin{align}\label{stein.prior}
 \pi_S(\mu)=\|\mu\|^{2-d}
\end{align}
which do satisfy Brown's conditions do not satisfy Brown-Hwang's.
\cite{Maruyama-Takemura-2008} and \cite{Maruyama-2009b} somewhat enlarged the class of admissible generalized Bayes estimators when the prior density of $\mu$ is spherically symmetric. 
However, these three papers \cite{Brown-Hwang-1982}, \cite{Maruyama-Takemura-2008} and \cite{Maruyama-2009b} 
do not seriously consider the boundary between admissibility and inadmissibility, which \cite{Brown-1971} does.

In this review paper, we assume the prior density is given by
\begin{align}\label{prior.pimu}
\pi(\mu)= \int_0^\infty g^{-d/2}\exp\left(-\frac{\|\mu\|^2}{2g}\right)\Pi(\rd g),
\end{align}
with non-negative measure $\Pi$ on $g$, and are going to determine admissibility/inadmissibility
of generalized Bayes estimators for a certain subclass of mixture priors in terms of $\Pi$  with a completely self-contained proof.
If $\Pi$ is finite or proper, the corresponding Bayes estimator is admissible.
Thus we are mainly interested in the case of infinite $\Pi$. However,
we do not exclude the case of a finite measure $\Pi$.

By the identity
\begin{align*}
 \|x-\mu\|^2+\frac{\|\mu\|^2}{g}=\frac{g+1}{g}\left\|\mu-\frac{g}{g+1}x\right\|^2+\frac{\|x\|^2}{g+1},
\end{align*}
the marginal likelihood $m_\pi$ is 
\begin{equation}\label{mpi_marginal}
m_\pi(\|x\|^2)= \int_{\mathbb{R}^d}\phi(x-\mu)\pi(\mu)\rd \mu 
=\int_0^\infty(g+1)^{-d/2}\exp\left(-\frac{\|x\|^2}{2(g+1)}\right)\Pi(\rd g),
\end{equation}
which is finite for all $x$ if
\begin{align}\label{minimum}
 \int_0^\infty\frac{\Pi(\rd g)}{(g+1)^{d/2}}<\infty.
\end{align}
Throughout the paper, we assume the prior $\Pi$ satisfies \eqref{minimum}.
By \eqref{bayes.esti.F}, together with $m_\pi$ given in \eqref{mpi_marginal},
the generalized Bayes estimator under $\pi(\mu)$ given by \eqref{prior.pimu} is written as
\begin{align}\label{estimator_pi}
 \delta_\pi=x+\nabla_x\log m_\pi(\|x\|^2)
=\left(1-\frac{\int_0^\infty(g+1)^{-d/2-1}\exp\left(-\frac{\|x\|^2}{2(g+1)}\right)\Pi(\rd g)}
{\int_0^\infty(g+1)^{-d/2}\exp\left(-\frac{\|x\|^2}{2(g+1)}\right)\Pi(\rd g)}\right)x.
\end{align}

In earlier studies, the prior \eqref{prior.pimu} with some regular varying mixing density
has been used in order to establish minimaxity or both minimaxity and admissibility 
of the estimator \eqref{estimator_pi}. See Remark \ref{rem:piabc}.
The class of mixing priors we consider is the following generalization of such mixing densities. 

Suppose $\Pi(\rd g)$ in \eqref{prior.pimu} has a regularly varying density 
of the form
\begin{align}
 \pi(g;a,b,c)=(g+1)^a\left(\frac{g}{g+1}\right)^b\left\{\log(g+1)+1\right\}^c,\label{pigabc}\\
\text{for }a<d/2-1,\quad b>-1,\quad c\in\mathbb{R},\label{a.b.c}
\end{align}
where \eqref{a.b.c} is necessary and sufficient for \eqref{minimum}.
A Tauberian theorem (see, e.g., Theorem 13.5.4 in \citet{Feller-1971}) gives
\begin{align}\label{tauberian}
\lim_{t\to\infty} \frac{t^{d/2-1}m_{\pi}(t;a,b,c)}{\pi(t;a,b,c)}=\Gamma(d/2-1-a)2^{d/2-1-a}
\end{align}
where
\begin{align}\label{tauberian_marginal}
 m_{\pi}(\|x\|^2;a,b,c)=\int_0^\infty(g+1)^{-d/2}\exp\left(-\frac{\|x\|^2}{2(g+1)}\right)\pi(g;a,b,c)\rd g.
\end{align}
Hence the integrability (or non-integrability) of \eqref{Brown1971.infty.integral} 
of is equivalent to integrability (or non-integrability) of
\begin{align}
 \int_1^\infty \frac{\rd g}{g \pi(g;a,b,c)},
\end{align}
where, as in Lemma \ref{lem:integrability.1},
\begin{align*}
 \int_1^\infty\frac{\rd g}{g\pi(g;a,b,c)}
\begin{cases}
 < \infty & \text{either }a>0 \text{ or }\{a=0 \text{ and }c>1\}, \\
=\infty & \text{either }a<0 \text{ or }\{a=0 \text{ and }c\leq 1\}.
\end{cases}
\end{align*}
Further the risk of the corresponding (generalized) Bayes estimator is bounded since
\begin{align}
\sup_{x} \|\nabla_x\log m_\pi(\|x\|^2;a,b,c)\|^2<\infty,
\end{align}
which is shown in Lemma \ref{lem:bound.risk}.
Hence, given the result of \cite{Brown-1971}, we have a following result.
\begin{thm}\label{thm.intro}
Suppose the prior $\Pi(g)$ has a regularly varying density, $\pi(g;a,b,c)$ given by \eqref{pigabc}. 
Then admissibility/inadmissibility of the the corresponding (generalized) Bayes estimator 
with $\pi(g;a,b,c)$ is determined by 
non-integrability/integrability of
\begin{align}\label{eq:inad_integrability}
 \int_1^\infty\frac{\rd g}{g\pi(g;a,b,c)},
\end{align} 
or equivalently, of \eqref{Brown1971.infty.integral}.
\end{thm}
In this paper, we will provide a self-contained proof of this result.
Further, our more general sufficient condition for admissibility in Theorem \ref{thm:ad.case.i},
\begin{align}
 \int_0^\infty \frac{\Pi(\rd g)}{g+1}  <\infty
\end{align}
includes the case where the risk of the estimator is not bounded. See Remark \ref{rem:unbounded.risk}.

The organization of this paper is as follows. 
Section \ref{sec:inad} deals with inadmissibility and 
follows a technique used in \cite{Dasgupta-Strawderman-1997}. 
Section \ref{sec:admissi} is concerned with admissibility. 
An appendix is devoted to technical results needed in the development.

\section{Inadmissibility}
\label{sec:inad}
For the inadmissibility part of Brown's theorem, the following proof is essentially due to \cite{Dasgupta-Strawderman-1997}, which relates inadmissibility to solving Riccati differential equations.
Following \cite{Brown-1971} and \cite{Dasgupta-Strawderman-1997}, 
we do not assume \eqref{prior.pimu} but just spherical symmetry on $\Pi$.
Hence the result is presented in terms of $m_\Pi$ \eqref{m_F} and $\delta_{\Pi}$ \eqref{bayes.esti.F}, 
not  $m_\pi$ \eqref{mpi_marginal} and $\delta_\pi$ \eqref{estimator_pi}.
\begin{thm}\label{thm.inad}
 The generalized Bayes estimator $\delta_{\Pi}$ under the spherically symmetric prior $\Pi(\rd \mu)$
given by \eqref{bayes.esti.F},
is inadmissible if
\begin{align}\label{inad.suff}
\int_1^\infty\frac{\rd t}{t^{d/2}m_{\Pi}(t)} <\infty,
\end{align}
where $m_{\Pi}(\|x\|^2):= m_{\Pi}(x)$.
\end{thm}
Notice that the statement of Theorem \ref{thm.inad} is equivalent to the inadmissibility condition of \cite{Brown-1971}.
Hence, by  \eqref{tauberian},
the inadmissibility part of Theorem \ref{thm.intro} follows.
More concretely, by Lemma \ref{lem:integrability.1}, the integrability \eqref{eq:inad_integrability}
corresponds to the case either $a>0$ or \{$a=0$ and $c>1$\} and hence
we have a following corollary.
\begin{corollary}\label{cor:inad}
The generalized Bayes estimator with the mixing density $\pi(g;a,b,c)$ 
is inadmissible if either $a>0$ or \{$a=0$ and $c>1$\}.
For these values of $a$ and $c$, the integral 
\eqref{Brown1971.infty.integral} and 
\begin{align}
 \int_1^\infty\frac{\rd g}{g\pi(g;a,b,c)}
\end{align} 
converges.
\end{corollary}

\medskip

\begin{proof}[Proof of Theorem \ref{thm.inad}]
By the \cite{Stein-1974} identity, the risk function of an estimator of
the form 
\begin{align}
 \delta_h(x)=x+h(x)=(x_1+h_1(x), \dots , x_d+h_d(x))^\T
\end{align}
is given by
\begin{align*}
 R(\mu,\delta_h)&= 
E\left[\|\delta_h(X)-\mu\|^{2} \right] \\
&= E\left[ \|X-\mu\|^2\right]
+E_{\mu}\left[\|h(X)\|^{2}\right]
+2\sum_{i=1}^d
E_{\mu}\left[(X-\mu)^\T h(X)\right]  \\ 
 &= E\big[\hat{R}_h(X)\big],
\end{align*}
where $\hat{R}_h(x)$ is called the SURE (Stein Unbiased Risk Estimate) and is given by
\begin{align}\label{hat.R.g}
\hat{R}_h(x)=d+ \|h(x)\|^{2}+2\sum_{i=1}^{d}\frac{\partial}{\partial x_{i}}h_{i}(x).
\end{align}
With $ h(x)=\nabla \log m_{\Pi} (\| x \|^{2}) =2 xm'_{\Pi} (\|x\|^2)/m_{\Pi} (\|x\|^2)$ in \eqref{hat.R.g},
the SURE of $\delta_{\Pi} $ given by \eqref{bayes.esti.F}, is
\begin{equation} 
\hat{R}_{\Pi}(\|x\|^2) = d-4w\left(\frac{m'_{\Pi} (w)}{m_{\Pi} (w)}\right)^{2}
+4d\frac{m'_{\Pi} (w)}{m_{\Pi} (w)}+8w\frac{m''_{\Pi} (w)}{m_{\Pi} (w)},
\end{equation}
where $ w=\|x\|^{2}$.
Similarly the SURE of the estimator
\begin{align}
 \delta_{\Pi,k}=\delta_{\Pi}  -2\frac{k(\|x\|^2)}{m_{\Pi} (\|x\|^2)}x
\end{align}
is
\begin{align}
 \hat{R}_{\Pi,k}(\|x\|^2)=\hat{R}_{\Pi} (\|x\|^2)+\Delta(\|x\|^2)
\end{align}
where
\begin{align}
\Delta(w)=4w\frac{k^2(w)}{m_{\Pi} (w)}\left(\frac{1}{m_{\Pi} (w)}-\frac{dk(w)+2wk'(w)}{w k^2(w)}\right).
\end{align}
Now let $ q(w)=w^{d/2}k(w)$. Then
\begin{align}
 \frac{1}{m_{\Pi} (w)}-\frac{dk(w)+2wk'(w)}{w k^2(w)}=2w^{d/2}\left\{\frac{1}{2w^{d/2}m_{\Pi} (w)}+
\frac{\rd}{\rd w}\left(\frac{1}{q(w)}\right)\right\}.
\end{align}
Assume
\begin{align}\label{inad}
c= \int_1^\infty\frac{\rd t}{t^{d/2}m_{\Pi} (t)}<\infty.
\end{align}
Then a solution of the differential equation $\Delta(w)=0$ is given by
\begin{align}
 \frac{1}{q_*(w)}= -\frac{1}{2}\int_1^w \frac{\rd t}{t^{d/2}m_{\Pi} (t)}+c.
\end{align}
Hence, under \eqref{inad}, the estimator
\begin{align}
\delta_{\Pi,k_*}=\delta_{\Pi}  -2\frac{k_*(\|x\|^2)}{m_{\Pi} (\|x\|^2)}x
\end{align}
with $k_*(w)=w^{-d/2}q_*(w)$ has the same risk as that of $ \delta_{\Pi} $.
Since quadratic loss is strictly convex, the estimator given by the  average of $ \delta_{\Pi} $ and $ \delta_{\Pi,k_*}$,
\begin{align}
 \delta_{\Pi}  - \frac{k_*(\|x\|^2)}{m_{\Pi} (\|x\|^2)}x, \label{average_estimator}
\end{align}
strictly improves on $ \delta_{\Pi} $.
\end{proof}
\begin{remark}
 Note
\begin{align}
 \lim_{w\to 0}w k_*(w)
&=\lim_{w\to 0}\frac{w^{-d/2+1}}{-(1/2)\int_1^w \rd t/\{ t^{d/2}m_{\Pi} (t)\}+c} \\
&=\lim_{w\to 0}\frac{(-d/2+1)w^{-d/2}}{-1/\{2 w^{d/2}m_{\Pi} (w)\}} \\
&=(d-2)m_{\Pi} (0).
\end{align}
Then the shrinkage factor of $ \delta_{\Pi} -\{k_*(\|x\|^2)/m_{\Pi} (\|x\|^2)\}x $, defined by
\begin{align}
1+2\frac{m'_{\Pi}(\|x\|^2)}{m_{\Pi}(\|x\|^2)}
-\frac{k_*(\|x\|^2)}{m_{\Pi} (\|x\|^2)},
\end{align}
approaches $-\infty$ as $\|x\|^2\to 0$, which implies that
the average estimator \eqref{average_estimator} is dominated by its positive-part estimator.
Hence the average estimator improves on $ \delta_{\Pi} $, but is still inadmissible.

Generally speaking, for a generalized Bayes and inadmissible estimator $\delta_{\Pi}$,
it is difficult to find an admissible estimator which dominates $\delta_{\Pi}$.
Interestingly, on page 863 of \cite{Brown-1971}, there is some related discussion on this topic.
\end{remark}

\section{Admissibility}
\label{sec:admissi}
Admissibility of (generalized) Bayes estimators is covered in this section. 
Subsection \ref{sec:Blyth} gives the form of Blyth's method used, 
while Subsection \ref{sec:case_I} gives an admissibility result under general $\Pi(\rd g)$. 
Subsection \ref{sec:case_II} considers slowly varying mixing density, $\pi(g;a,b,c)$ with $a=0$,
in order to study the boundary between admissibility and inadmissibility.

\subsection{Blyth's method}
\label{sec:Blyth}
In Subsection \ref{sec:case_I} a form of Blyth's method  (\cite{Blyth-1951}),
 given in Lemma
\ref{lem.Blyth} 
below, is applied to establish admissibility of a class of generalized Bayes estimators. It utilizes a sequence of  proper priors of the form
\begin{align*}
\pi_i(\mu)=\int_0^\infty 
g^{-d/2}\exp\left(-\frac{\|\mu\|^2}{2g}\right)h^2_i(g)\Pi(\rd g)
\end{align*}
 where $ h_i(g)$ satisfies $\int_0^\infty h^2_i(g)\Pi(\rd g)<\infty$ for any fixed $i$
and $\lim_{i\to\infty}h_i(g)=1$ for any fixed $g$.
Specific choice of $h_i$ will be given by \eqref{h_i_1} in Section \ref{sec:case_I} and \eqref{h_i_g} in Appendix \ref{sec:proof}.

Under the  prior $\pi_i(\mu)$, we have
\begin{align*}
 m_i(\|x\|^2)=
\int_{\mathbb{R}^d}\phi(x-\mu)\pi_i(\mu)\rd \mu =
\int_0^\infty(g+1)^{-d/2}\exp\left(-\frac{\|x\|^2}{2(g+1)}\right)h^2_i(g)\Pi(\rd g)
\end{align*}
and
\begin{equation}\label{delta_i}
 \delta_i=x+ \nabla_x\log m_i(\|x\|^2)=\left(1-\frac{\int_0^\infty(g+1)^{-d/2-1}\exp\left(-\frac{\|x\|^2}{2(g+1)}\right)h^2_i(g)\Pi(\rd g)}{\int_0^\infty(g+1)^{-d/2}\exp\left(-\frac{\|x\|^2}{2(g+1)}\right)h^2_i(g)\Pi(\rd g)}\right)x.
\end{equation}
The Bayes risk difference between $\delta_\pi$ and $\delta_i$, with respect to $\pi_i(\mu)$, is
\begin{align*}
 \Delta_i=\int \left\{R(\delta_\pi,\mu)-R(\delta_i,\mu)\right\}\pi_i(\mu)\rd \mu,
\end{align*}
which is rewritten as
\begin{align}
 \Delta_i &= 
 \int_{\mathbb{R}^d} \int_{\mathbb{R}^d} \left\{ \|\delta_\pi - \mu\|^2 - \|\delta_i - \mu\|^2 
\right\} \phi(x-\mu)\pi_i(\mu)\rd \mu \rd x \\
&= 
\int_{\mathbb{R}^d} \bigg\{ \big( \| \delta_\pi \|^2 -\| \delta_i\|^2\big) m_i(\|x\|^2)
-2(\delta_\pi -\delta_i)^\T \int_{\mathbb{R}^d} \mu \phi(x-\mu)\pi_i(\mu)\rd \mu \bigg\}\rd x \\
 &= \int_{\mathbb{R}^d} \|  \delta_\pi - \delta_i  \|^{2} m_i(\|x\|^2) \rd x.\label{general.Delta.i}
\end{align}
The following form of Blyth's sufficient condition shows that 
$\lim_{i\to\infty}\Delta_i=0$ implies admissibility. The result applies to general sequences of prior densities under the given conditions and not just those constructed as above.
\begin{lem}\label{lem.Blyth}
Suppose $ \pi_i(\mu)$ is an increasing (in $i$) sequence of proper priors, 
$\lim_{i\to\infty}\pi_i(\mu)=\pi(\mu)$ and 
$\pi_i(\mu)>0$ for all $\mu$. Then $\delta_\pi$ is admissible if $\Delta_i$ satisfies 
$\lim_{i\to\infty}\Delta_i=0$.
\end{lem} 
\begin{proof}
  Suppose that $ \delta_\pi$ is inadmissible and hence that there exists a  $\delta'$ satisfies
  \begin{equation}\label{eq.weak}
R(\mu,\delta') \leq R(\mu,\delta_{\pi })
  \end{equation}
for all $\mu$ and
\begin{equation}\label{eq.strict}
R(\mu,\delta') < R(\mu,\delta_{\pi }) \text{ for some }\mu_0.
\end{equation}
By \eqref{eq.strict}, we have
\begin{align*}
 \int_{\mathbb{R}^d} \| \delta_\pi(x)-\delta'(x)\|^2 \phi(x-\mu_0)\rd x>0.
\end{align*}
Further we have
\begin{align*}
\int_{\mathbb{R}^d} \| \delta_\pi(x)-\delta'(x)\|^2 \phi(x-\mu)\rd x
=\int_{\mathbb{R}^d} \| \delta_\pi(x)-\delta'(x)\|^2 \frac{\phi(x-\mu)}{\phi(x-\mu_0)}\phi(x-\mu_0)\rd x.
\end{align*}
Since the ratio $\phi(x-\mu)/\phi(x-\mu_0)$
is continuous in $x$ and positive, it follows that
\begin{align*}
\int_{\mathbb{R}^d} \| \delta_\pi(x)-\delta'(x)\|^2 \phi(x-\mu)\rd x>0
\end{align*}
for all $\mu$.

Set $\delta''=(\delta_\pi+\delta')/2$. Then we have
\begin{align*}
 \|\delta''-\mu\|^2=\frac{\|\delta_\pi-\mu\|^2+\|\delta'-\mu\|^2}{2}
-\frac{\|\delta_\pi-\delta'\|^2}{4},
\end{align*}
and
\begin{align*}
R(\mu,\delta'') 
&=E\left(\|\delta''-\mu\|^2\right) \\
&<(1/2)E\left(\|\delta'-\mu\|^2\right) +(1/2)E\left(\|\delta_\pi -\mu\|^2\right)\\
 &= \frac{1}{2}\left\{R(\mu,\delta')+R(\mu,\delta_\pi)\right\} \\
 &\leq  R(\mu,\delta_\pi),
\end{align*}
for all $\mu$.
Then we have
\begin{align*}
\Delta_i
&=\int_{\mathbb{R}^d} \left\{
R(\mu,\delta_{\pi })-R(\mu,\delta_i) \right\}
\pi_i(\mu)
 \rd  \mu \\
 & \geq \int_{\mathbb{R}^d} \left\{
R(\mu,\delta_{\pi })-R(\mu,\delta'') \right\}
\pi_i(\mu)
 \rd  \mu \\
&\geq \int_{\mathbb{R}^d} \left\{
R(\mu,\delta_{\pi })-R(\mu,\delta'') \right\}
\pi_1(\mu)
 \rd  \mu \\
&>0
\end{align*}
which contradicts $ \Delta_i \to 0$ as $i\to\infty$.
\end{proof}
Note that the integrand of $\Delta_i$ in \eqref{general.Delta.i} tends to $0$ as $i$ tends to infinity. 
Hence, in using the lemma to show admissibility, the bulk of the remainder of the proof consists in showing that the integrand is bounded by an integrable function. 
Then, by the dominated convergence theorem, $\lim_{i\to\infty}\Delta_i=0$ is satisfied so that $\delta_\pi$ is admissible.
Establishment of this dominating function constitutes much of the technical development for the admissibility results given in Sections \ref{sec:case_I} and \ref{sec:case_II}.

\subsection{A general admissibility result for mixture priors}
\label{sec:case_I}
This subsection is devoted to establishing the following result.
\begin{thm}\label{thm:ad.case.i}
The (generalized)  Bayes estimator $\delta_\pi$ given by \eqref{estimator_pi} is admissible if
\begin{align}\label{brown.simple}
 \int_0^\infty \frac{\Pi(\rd g)}{g+1}  <\infty.
\end{align}
\end{thm}
Clearly any proper prior on $g$ satisfies \eqref{brown.simple}. Further, even if the prior is improper, i.e., $ \int_0^\infty \Pi(\rd g)=\infty$, 
the corresponding generalized Bayes estimator is admissible under \eqref{brown.simple}.
Further, by Lemma \ref{lem:integrability.1}, 
we have a following corollary.
\begin{corollary}\label{cor:ad}
The (generalized) Bayes estimator with mixing density $\pi(g;a,b,c)$ given by
\eqref{pigabc}
is admissible if either $a<0$ or \{$a=0$ and $c<-1$\}. 
For these values of $a$ and $c$, the integral 
\ref{Brown1971.infty.integral} and 
\begin{align}
 \int_1^\infty\frac{\rd g}{g\pi(g;a,b,c)}
\end{align} 
diverges.
\end{corollary}

\medskip

\begin{proof}[Proof of Theorem \ref{thm:ad.case.i}]
Let
\begin{align}\label{h_i_1}
 h_i^2(g)=\frac{i}{g+i},
\end{align}
which is increasing in $i$ and is such that $\lim_{i\to\infty}h_i(g)=1$, for any fixed $g$.
The prior $\pi_i(\mu)$ is proper for any fixed $i$ since the mixture distribution is proper, because
\begin{align}
 \int_0^\infty h^2_i(g)\Pi(\rd g)\leq i\int_0^\infty \frac{\Pi(\rd g)}{g+1}<\infty.
\end{align}
The integrand of $\Delta_i$ given in \eqref{general.Delta.i} is 
\begin{align}
& \|  \delta_\pi- \delta_i  \|^{2} m_i(w) \label{del-deli*m}\\
&= w \left(\frac{\int_0^\infty (g+1)^{-1}F( w ,g)h_i^2(g)\Pi(\rd g)}
{\int_0^\infty F( w ,g) h_i^2(g)\Pi(\rd g)}
-\frac{\int_0^\infty(g+1)^{-1}F( w ,g)\Pi(\rd g)}{\int_0^\infty F( w ,g)\Pi(\rd g)}\right)^2
\int_0^\infty F( w ,g)h_i^2(g)\Pi(\rd g),
\end{align}
where $w=\|x\|^2$ and 
\begin{align}\label{eq:Fwg}
 F( w ,g)=(g+1)^{-d/2}\exp\left(-\frac{ w }{2(g+1)}\right).
\end{align}
Applying
the inequality
\begin{equation}\label{eq.4}
\left(\sum_{i=1}^k a_i\right)^2\leq k \sum_{i=1}^k a_i^2,
\end{equation}
for $k\in\mathbb{Z}$ to \eqref{del-deli*m}, we have
\begin{align}
& \|  \delta_\pi- \delta_i  \|^{2} m_i( w ) \\
& \leq 2 w \left( 
\frac{\{\int_0^\infty (g+1)^{-1}F( w ,g)h_i^2(g)\Pi(\rd g)\}^2}
{\int_0^\infty F( w ,g) h_i^2(g)\Pi(\rd g)} 
+
\frac{\{\int_0^\infty (g+1)^{-1}F( w ,g) \Pi(\rd g)\}^2}
{\int_0^\infty F( w ,g)\Pi(\rd g)}\right). \label{eq:case_I_1}
\end{align}
Further applying the Cauchy-Schwarz inequality to the first and second terms of \eqref{eq:case_I_1}, 
we have
\begin{align*}
\| \delta_\pi- \delta_i \|^{2} m_i( w )  \leq 
4 w \int_0^\infty \frac{F( w ,g)}{(g+1)^{2}}\Pi(\rd g).
\end{align*}
This is precisely the bound required in order to apply the dominated convergence theorem to demonstrate that $\Delta_{i} \to 0$, since
\begin{align*}
 \int_{\mathbb{R}^d}\|  \delta_\pi- \delta_i  \|^{2} m_i(\|x\|^2)\rd x 
&\leq 4\int_{\mathbb{R}^d}\int_0^\infty\frac{\|x\|^2}{(g+1)^{d/2+2}}\exp\left(-\frac{\|x\|^2}{2(g+1)}\right)
\Pi(\rd g)\rd x  \\
&=
4\int_{\mathbb{R}^d}\|y\|^2 \exp\left(-\frac{\|y\|^2}{2}\right)\rd y
\int_0^\infty \frac{\Pi(\rd g)}{g+1} \\ &<\infty,
\end{align*}
which completes the proof.
\end{proof}

\begin{remark}\label{rem:unbounded.risk}
The prior $\mu\sim N_d(0,gI_d)$ corresponds to the point prior on $g$ in \eqref{prior.pimu}. 
The proper Bayes estimator is 
\begin{align*}
 \frac{g}{g+1}X
\end{align*}
 with unbounded risk 
\begin{align*}
\frac{g^2 d+\|\mu\|^2}{(g+1)^2}.
\end{align*}
Theorem \ref{thm:ad.case.i} covers this case 
whereas \eqref{Brown1971.infty} of \cite{Brown-1971}'s result is not satisfied by the proper Bayes admissible estimator $gX/(g+1)$.
\end{remark}

\subsection{On the boundary between admissibility and inadmissibility}
\label{sec:case_II}
For the class of densities $\pi(g;a,b,c)$, Corollaries \ref{cor:inad} and \ref{cor:ad} settle the issue of admissibility/inadmissibility for all values of $a$ and $c$ except for the cases \{$a=0$ and $|c|\leq 1$\}, where, in particular,  the case $a=0$, $b=0$, $c=0$ is corresponding to 
the \cite{Stein-1974} prior given by \eqref{stein.prior} since 
\begin{align}
\int_0^\infty g^{-d/2}\exp\left(-\frac{\|\mu\|^2}{2g}\right)\rd g=
\Gamma(d/2-1)2^{d/2-1}\|\mu\|^{2-d}.
\end{align}
For the cases \{$a=0$ and $|c|\leq 1$\}, the non-integrability
\begin{align}
 \int_1^\infty\frac{\rd g}{g\pi(g;a,b,c)}=\infty
\end{align}
follows.
Theorem \ref{thm.intro} shows that these values, near the boundary between admissibility and inadmissibility correspond to admissibility.
Recall that Theorem \ref{thm.intro} is a corollary of \cite{Brown-1971}.
Here, we  provide a self-contained proof of the admissibility result.
\begin{thm}\label{thm:main}
Assume the measure $\Pi(\rd g)$ has the density $\pi(g;a,b,c)$ with $a=0$, $b\geq 0$ and $|c|\leq 1$.
Then the corresponding generalized Bayes estimator is admissible.
\end{thm}
Proof of Theorem \ref{thm:main} will be given in Section \ref{sec:proof} of the Appendix.
Unlike Corollaries \ref{cor:inad} and \ref{cor:ad}, we assume $b\geq 0$ in Theorem \ref{thm:main},
which is needed for an integration by parts. See \eqref{parts} in Appendix for the details.

\begin{remark}\label{rem:piabc}
Here is a remark on the regular varying mixing density $\pi(g;a,b,c)$ given by \eqref{pigabc},
\begin{align}
 \pi(g;a,b,c)=(g+1)^a\left(\frac{g}{g+1}\right)^b\left\{\log(g+1)+1\right\}^c.
\end{align}
The term $(g+1)^a$ comes from \cite{Strawderman-1971} and \cite{Berger-1976}, whereas
$\{g/(g+1)\}^b$ comes from \cite{Faith-1978}.
In these three papers, conditions for minimaxity or both minimaxity and admissibility 
of the (generalized) Bayes estimator \eqref{estimator_pi} have been established. 

In this paper, we introduce the term $\left\{\log(g+1)+1\right\}^c$, mainly in order to
clarify the boundary between admissibility and inadmissibility under $a=0$.
Additionally, by Theorem 1 of \cite{Fourdrinier-etal-1998}, we have a following result on minimaxity based on
\begin{align}
(g+1)\frac{\{\rd /\rd g\} \pi(g;a,b,c)}{\pi(g;a,b,c)}=a+\frac{b}{g}+\frac{c}{\log(g+1)+1}.
\end{align}
\end{remark}

\begin{thm}
Assume $ -d/2+1 + \max(0,-2c)\leq a <d/2-1$ and $b\geq 0$. 
Then the (generalized) Bayes estimator with the mixing density $\pi(g;a,b,c)$ is minimax.
\end{thm}

\section{Concluding remarks}\label{Conclusion}
We have revisited Brown's monumental 1971, paper on admissibility, and have given an independent proof of one of his main results for the subclass of scale mixture priors with mixing density of the form 
\begin{align}
 \pi(g;a,b,c)=(g+1)^a\left(\frac{g}{g+1}\right)^b\left\{\log(g+1)+1\right\}^c.
\end{align}
For this class of densities, 
we have established that Brown's condition for admissibility/inadmissibility, based on the  non-integrability/integrability  on $[1, \infty)$ of  $1/\{t^{p/2}m(t)\}$ 
(where $m(\|x\|^2)$ is the spherically symmetric marginal density) is equivalent to non-integrability/integrability of $1/\{g \pi(g)\}$ on $[1, \infty)$. 
We established a sharp bound between admissibility and inadmissibility for this class.
We also established some results, applicable to more general classes of mixing distributions, 
that may be of independent interest.
In particular, we have shown that integrability of $\Pi(\rd g)/(g+1)$ on $[0, \infty)$ suffices to establish admissibility for any generalized mixing prior distribution, 
independently of whether the resulting generalized Bayes estimator has bounded risk.

\appendix

\section{Proof of Theorem \ref{thm:main}}
\label{sec:proof}
In this appendix,  
it is convenient to use the notation
\begin{align*}
\pi(g)=\pi(g;0,b,c)= \left(\frac{g}{g+1}\right)^b\left\{\log(g+1)+1\right\}^c.
\end{align*}

\subsection{$h_i$}
Let 
\begin{align}\label{eq:Lg}
 L(g)=\log(g+1)+1.
\end{align}
Then the non-integrability 
\begin{align*}
 \int_0^\infty\frac{\rd g}{(g+1)L(g)} =\int_0^\infty \frac{\rd z}{z+1}=\infty
\end{align*}
follows.
Let
\begin{align}\label{h_i_g}
 h_i(g)=\begin{cases} 
	 \displaystyle 1-\frac{\int_0^g \rd t/\{(t+1)L(t)\}}{\int_0^i \rd t/\{(t+1)L(t)\}} & 0< g <i \\
	 0 & g\geq i.
	\end{cases}
\end{align}
For fixed $g$, $h_i(g)$ is increasing in $i$ and $\lim_{i\to\infty}h_i(g)=1$.
Since $h_i(g)=0$ for $ g\geq i$, 
$\int\pi(g)h_i^2(g)\rd g<\infty$ even if $ \int \pi(g)\rd g=\infty$.
Note $ h_i(g)$ is piecewise differentiable as
\begin{align*}
h'_i(g)=\begin{cases} 
	 \displaystyle -\frac{ 1}
{(g+1)L(g)\int_0^i \rd t/\{(t+1)L(t)\}} & 0< g <i \\
	 0 & g\geq i.
	\end{cases}
\end{align*}
Then $h_i^2(g)$ is continuously differentiable since 
$\{h_i^2(g)\}'=2h_i(g)h'_i(g)$ and $h_i(i)=0$.
Further we have
\begin{align}\label{h_i_no_sup}
\sup_i \left|h'_i(g)\right|\{(g+1)L(g)\}
=\begin{cases} 
 \displaystyle 
\frac{ 1}{\int_0^1 \rd t/\{(t+1)L(t)\}}
 & 0< g <1 \\
 \displaystyle 
\frac{ 1}{\int_0^g \rd t/\{(t+1)L(t)\}}
 &  g \geq 1,
	\end{cases}
\end{align}
which will be used in Section \ref{sec:A2}.

\subsection{re-expression of the risk difference}
Integration by parts of the numerator of $\delta_i$ given by \eqref{delta_i} gives 
\begin{align}
&\int_0^\infty(g+1)^{-d/2-1}\exp\left(-\frac{ w }{2(g+1)}\right)\pi(g)h_i^2(g)\rd g \label{parts}\\
&=\frac{2}{ w }\left[(g+1)^{-d/2+1}\exp\left(-\frac{ w }{2(g+1)}\right)\pi(g)h_i^2(g)\right]_0^\infty \\
&\quad + \frac{d-2}{ w }\int_0^\infty(g+1)^{-d/2}\exp\left(-\frac{ w }{2(g+1)}\right)\pi(g)h_i^2(g)\rd g 
\\ &\quad - \frac{2}{ w }\int_0^\infty(g+1)^{-d/2+1}\exp\left(-\frac{ w }{2(g+1)}\right)\pi'(g)h_i^2(g)\rd g \\
&\quad - \frac{4}{ w }\int_0^\infty(g+1)^{-d/2+1}\exp\left(-\frac{ w }{2(g+1)}\right)\pi(g)h_i(g)h'_i(g)\rd g \\
&=-\frac{2\pi(0)}{ w }\exp\left(-\frac{ w }{2}\right)+\frac{(d-2)m_i( w )}{ w } \\
&\quad - \frac{2}{ w }\int_0^\infty (g+1) F( w ,g)\pi'(g)h_i^2(g)\rd g 
 - \frac{4}{ w }\int_0^\infty(g+1) F( w ,g)\pi(g)h_i(g)h'_i(g)\rd g,
\end{align}
where $F(w,g)$ is given in \eqref{eq:Fwg}. 
For $\delta_\pi$, let $h_i\equiv 1$ in \eqref{parts}. 
Note
\begin{align}
 (g+1)\frac{\pi'(g)}{\pi(g)}=\frac{b}{g}+\frac{c}{L(g)},
\end{align}
where $L(g)$ is given by \eqref{eq:Lg}. Then
$\|  \delta_\pi-\delta_i  \|^2 m_i( w )$ given in \eqref{general.Delta.i} is 
\begin{align}
 \|  \delta_\pi-\delta_i  \|^2 m_i( w ) 
=4\frac{m_i( w )}{ w }\left\{\pi(0)A_1(w)-2A_2(w)+bA_3(w)+cA_4(w)\right\}^2,
\end{align}
where
\begin{align}
A_1(w)&=
\frac{\exp(- w /2)}{m_\pi( w )}-\frac{\exp(- w /2)}{m_i( w )}, \\
A_2(w)&=\frac{\int_0^\infty(g+1)F( w ,g)\pi(g)h_i(g)h'_i(g)\rd g}{m_i( w )}, \\
A_3(w)&=\frac{\int_0^\infty g^{-1}F( w ,g)\pi(g)\rd g}{m_\pi( w )}
-\frac{\int_0^\infty g^{-1}F( w ,g)\pi(g)h_i^2(g)\rd g}{m_i( w )}, \\
A_4(w)&=
\frac{\int_0^\infty \{1/L(g)\}F( w ,g)\pi(g)\rd g}{m_\pi( w )}
-\frac{\int_0^\infty\{1/L(g)\}F( w ,g)\pi(g)h_i^2(g)\rd g}{m_i( w )}.
\end{align}
Further, by the inequality \eqref{eq.4}, we have
\begin{align*}
 \|  \delta_\pi-\delta_i  \|^2 m_i( w ) 
\leq 16\frac{m_i( w )}{ w }\left\{\pi(0)^2A^2_1(w)+4A^2_2(w)+b^2A^2_3(w)+c^2A^2_4(w)\right\}.
\end{align*}
As noted earlier, the proof is completed by proving dominated convergence for each of these 4 terms.

\subsection{Dominated convergence for the term involving $A_1$}
\label{sec:A1}
Since $m_1( w )\leq m_i( w )\leq m_\pi( w )$, we have
\begin{align}
 m_i( w )A^2_1( w )\leq \frac{\exp(- w )}{  m_i( w )}
\leq \frac{\exp(- w )}{ m_1( w )}=\frac{\exp(- w/2 )}{\exp( w/2 ) m_1( w )}.
\end{align}
Then we have
\begin{align*}
\exp( w /2)m_1( w )&=
\int_0^\infty(g+1)^{-d/2}\exp( w /2)\exp\left(-\frac{ w }{2(g+1)}\right)\pi(g)h_1^2(g)\rd g \\
&=\int_0^\infty(g+1)^{-d/2}\exp\left(\frac{g w }{2(g+1)}\right)\pi(g)h_1^2(g)\rd g \\
&\geq \int_0^\infty(g+1)^{-d/2}\pi(g)h_1^2(g)\rd g ,
\end{align*}
and hence
\begin{align}
 m_i( w )A^2_1( w )\leq 
\frac{\exp(- w /2)}{\int_0^\infty(g+1)^{-d/2}\pi(g)h_1^2(g)\rd g }.
\end{align}
By Part \ref{lem:gamma.1} of Lemma \ref{lem:gamma}, we have
\begin{align}
\int_{\mathbb{R}^d} \frac{m_i(\|x\|^2)}{\|x\|^2}A^2_1(\|x\|^2)\rd x
\leq \frac{C_d}{\int_0^\infty(g+1)^{-d/2}\pi(g)h_1^2(g)\rd g}<\infty,
\end{align}
where
\begin{align}\label{C_d}
 C_d=\frac{\pi^{d/2}2^{d/2}}{d-2}.
\end{align}

\subsection{Dominated convergence for the term involving $A_2$}
\label{sec:A2}
By the Cauchy-Schwarz inequality, 
\begin{align*}
& \left(\int_0^\infty(g+1)F( w ,g)\pi(g)h_i(g)h'_i(g)\rd g\right)^2 \\
&\leq \int_0^\infty F( w ,g)\pi(g)h^2_i(g)\rd g
\int_0^\infty (g+1)^{2}F( w ,g)\pi(g)\{h'_i(g)\}^2\rd g \\
&=m_i( w )
\int_0^\infty (g+1)^{2}F( w ,g)\pi(g)\{h'_i(g)\}^2\rd g.
\end{align*}
Then,
\begin{align}
m_i(w)A^2_2(w)  
&\leq 
\int_0^\infty
(g+1)^{2}F(w ,g)\pi(g)\{h'_i(g)\}^2\rd g \label{eq.A_2.1}\\
&\leq 
\int_0^\infty
(g+1)^{2}F( w ,g)L(g)\{h'_i(g)\}^2\rd g,
\end{align}
where the second inequality follows from
the fact
\begin{align}
 \pi(g)=\left(\frac{g}{g+1}\right)^b\left\{\log(g+1)+1\right\}^c\leq L(g)
\end{align}
for $b\geq 0$ and $|c|\leq 1$.
By \eqref{h_i_no_sup} and Part \ref{lem:gamma.2} of Lemma \ref{lem:gamma}, we have
\begin{align}
& \frac{1}{C_d}\int_{\mathbb{R}^d} \frac{m_i( \|x\|^2 )}{ \|x\|^2 }A^2_2( \|x\|^2 )\rd x\\
&\leq 
\int_0^\infty(g+1)L(g)\sup_i\{h'_i(g)\}^2\rd g \\
&\leq 
\frac{1}{\{\int_0^1\rd t/\{(t+1)L(t)\}\}^2}
\int_0^1 \frac{\rd g}{(g+1)L(g)}
+\int_1^\infty\frac{1}{(g+1)L(g)}\frac{\rd g}{\{\int_0^g\rd t/\{(t+1)L(t)\}\}^2}
 \\ 
&=
\frac{1}{\int_0^1\rd t/\{(t+1)L(t)\}}
+\int_1^\infty \left\{\frac{\rd}{\rd u}\left(-\frac{1}{\int_0^u\rd t/\{(t+1)L(t)\}}\right)\right\}\rd u
 \\ 
&=\frac{2}{\int_0^1\rd t/\{(t+1)L(t)\}}<\infty,
\end{align}
where $C_d$ is given by \eqref{C_d}.

\subsection{Dominated convergence for the term involving $A_3$}
\label{sec:A3}
Note
\begin{align}
m_i( w )A^2_3(w)\leq 
2
\left(\frac{\{\int_0^\infty g^{-1}F( w ,g)\pi(g)\rd g\}^2}{m_\pi( w )}
+\frac{\{\int_0^\infty g^{-1}F( w ,g)\pi(g)h_i^2(g)\rd g\}^2}{m_i( w )} \right).
\end{align}
By the covariance inequality, we have
\begin{align}
& \int_0^\infty g^{-1}F( w ,g)\pi(g)h_i^2(g)\rd g \\
&= \int_0^\infty \left(\frac{g+1}{g}\right)(g+1)^{-p/2-1}\exp\left(-\frac{ w }{2(g+1)}\right)\pi(g)h_i^2(g)\rd g \\
&\leq \frac{\int_0^\infty \{(g+1)/g\}(g+1)^{-p/2-1}\pi(g)h_i^2(g)\rd g}
{\int_0^\infty (g+1)^{-p/2-1}\pi(g)h_i^2(g)\rd g} 
\int_0^\infty \frac{\pi(g)h_i^2(g)}{(g+1)^{p/2+1}}\exp\left(-\frac{ w }{2(g+1)}\right)\rd g.
\end{align}
Hence we have
\begin{align}
m_i( w )A^2_3(w) 
\leq 
2Q
\left(\frac{\{\int_0^\infty (g+1)^{-1}F( w ,g)\pi(g)\rd g\}^2}{m_\pi( w )}
+\frac{\{\int_0^\infty (g+1)^{-1}F( w ,g)\pi(g)h_i^2(g)\rd g\}^2}{m_i( w )} \right),\label{eq.A_3.1}
\end{align} 
where
\begin{align}
Q=\left(\frac{\int_0^\infty \{(g+1)/g\}(g+1)^{-p/2-1}\pi(g)\rd g}
{\int_0^\infty (g+1)^{-p/2-1}\pi(g)h_1^2(g)\rd g}\right)^2.
\end{align}
Applying the Cauchy-Schwarz inequality to 
\eqref{eq.A_3.1}, we have
\begin{align}\label{eq.A_3.2}
m_i( w )A^2_3(w)\leq 
4Q
\int_0^\infty (g+1)^{-2}F( w ,g)\pi(g)\rd g.
\end{align}
By Part \ref{lem:gamma.2} of Lemma \ref{lem:gamma} and \eqref{eq.A_3.2}, we have
\begin{align}
\frac{1}{C_d} \int_{\mathbb{R}^d} \frac{m_i(\|x\|^2)}{\|x\|^2}A^2_3(\|x\|^2)\rd x\leq 
4Q\int_0^\infty \frac{\pi(g)}{(g+1)^{3}}\rd g<\infty,
\end{align}
where $C_d$ is given by \eqref{C_d}.

\subsection{Dominated convergence for the term involving $A_4$}
\label{sec:A4}
Recall
\begin{align}
 A_4(w)=
\frac{\int_0^\infty \{1/L(g)\}F( w ,g)\pi(g)\rd g}{m_\pi( w )}
-\frac{\int_0^\infty\{1/L(g)\}F( w ,g)\pi(g)h_i^2(g)\rd g}{m_i( w )},
\end{align}
where
\begin{align*}
 L(g)=\log(g+1)+1.
\end{align*}
Then we have
\begin{align}
m_i( w )A^2_4(w)
=\frac{1}{m_i( w ) } 
\left(\int_0^\infty F( w ,g)\pi(g)h_i^2(g)\left\{\frac{1}{L(g)}-
\frac{\int_0^\infty \{1/L(g)\}F( w ,g)\pi(g)\rd g}{m_\pi( w )}
\right\}\rd g\right)^2.
\end{align}
By the Cauchy-Schwarz inequality, we have
\begin{align}
m_i( w )A^2_4(w)
&\leq
\int_0^\infty F( w ,g)\pi(g)h_i^2(g)\left\{\frac{1}{L(g)}-
\frac{\int_0^\infty \{1/L(g)\}F( w ,g)\pi(g)\rd g}{m_\pi( w )}
\right\}^2\rd g \\
&\leq
\int_0^\infty F( w ,g)\pi(g)\left\{\frac{1}{L(g)}-
\frac{\int_0^\infty \{1/L(g)\}F( w ,g)\pi(g)\rd g}{m_\pi( w )}\right\}^2\rd g \\
&=
\int_0^\infty F( w ,g)\frac{\pi(g)}{L^2(g)}\rd g  -
\frac{\{\int_0^\infty\{1/L(g)\}F( w ,g)\pi(g)\rd g\}^2}{\int_0^\infty F( w ,g)\pi(g)\rd g},
\label{eq.A_4.1}
\end{align}
where the second inequality follows from the fact $h_i^2(g)\leq 1$.

When $-1\leq c<1$,
we have
\begin{align*}
 m_i( w )A^2_4(w)\leq \int_0^\infty F( w ,g)\frac{\pi(g)}{L^2(g)}\rd g 
\leq \int_0^\infty \frac{F( w ,g)}{L^{2-c}(g)}\rd g 
\end{align*}
and, by Part \ref{lem:gamma.2} of Lemma \ref{lem:gamma},
\begin{align*}
 \frac{1}{C_d} \int_{\mathbb{R}^d} \frac{m_i(\|x\|^2)}{\|x\|^2}A^2_4(\|x\|^2)\rd x
\leq \int_0^\infty \frac{\rd g}{(g+1)L^{2-c}(g)}=
\int_0^\infty \frac{\rd z}{(z+1)^{2-c}}=\frac{1}{1-c}.
\end{align*}

When $c=1$, we need a more careful treatment.
In \eqref{eq.A_4.1} we have
\begin{align*}
& \int_0^\infty\frac{F( w ,g)\pi(g)}{L(g)}\rd g\\
&= \int_0^\infty (g+1)^{-d/2}\exp\left(-\frac{ w }{2(g+1)}\right)\left(1-\frac{1}{g+1}\right)^b\rd g \\
&\geq \int_0^\infty (g+1)^{-d/2}\exp\left(-\frac{ w }{2(g+1)}\right)\left(1-\frac{\max(b,1)}{g+1}\right)\rd g \\
&= w^{-d/2+1}\int_0^w t^{d/2-2}\exp(-t/2)\left(1-t\frac{\max(b,1)}{w}\right)\rd t \\
&\geq w^{-d/2+1}\left(2^{d/2-1}\Gamma(d/2-1)-\int_w^\infty t^{d/2-2}\exp(-t/2)\rd t -\max(b,1)
\frac{2^{d/2}\Gamma(d/2)}{w}\right). 
\end{align*}
Hence there exist $Q_1>0$ and $w_1>\exp(2Q_1)$ such that
\begin{align}\label{eq.A_4.2}
\int_0^\infty\frac{F( w ,g)\pi(g)}{L(g)}\rd g
\geq \frac{2^{d/2-1}\Gamma(d/2-1)}{w^{d/2-1}}\left(1- \frac{Q_1}{L(w)}\right)
\geq \frac{1}{2}\frac{2^{d/2-1}\Gamma(d/2-1)}{w^{d/2-1}},
\end{align}
for all $w\geq w_1$. 
Further we have
\begin{align}
& \int_0^\infty F( w ,g)\pi(g)\rd g\\
&= \int_0^\infty (g+1)^{-d/2}\exp\left(-\frac{ w }{2(g+1)}\right)\left(1-\frac{1}{g+1}\right)^b
\{\log(g+1)+1\}\rd g \\
&= \int_0^\infty (g+1)^{-d/2}\exp\left(-\frac{ w }{2(g+1)}\right)\left(1-\frac{1}{g+1}\right)^b
\{\log w - \log \{w/(g+1)\} +1\}\rd g \\
&= (\log w +1) \int_0^\infty\frac{F( w ,g)\pi(g)}{L(g)}\rd g
-
w^{-d/2+1}\int_0^w (\log t) t^{d/2-2}\exp(-t/2)(1-t/w)^b\rd t\\
&\leq L(w) \int_0^\infty\frac{F( w ,g)\pi(g)}{L(g)}\rd g+w^{-d/2+1}\int_0^\infty|\log t|t^{d/2-2}\exp(-t/2)\rd t.\label{eq.A_4.2.5}
\end{align}
By \eqref{eq.A_4.2} and \eqref{eq.A_4.2.5}, for all $w\geq w_1$, we have
\begin{align}\label{eq.A_4.2.6}
 \frac{\int_0^\infty F( w ,g)\pi(g)\rd g}{\int_0^\infty\{1/L(g)\}F( w ,g)\pi(g)\rd g}
\leq L(w)+Q_2
\end{align}
where
\begin{align}
 Q_2=2\frac{\int_0^\infty|\log t|t^{d/2-2}\exp(-t/2)\rd t}{2^{d/2-1}\Gamma(d/2-1)}.
\end{align}
Further, for all $w\geq \max(w_1,\exp(Q_2))$, we have $Q_2/L(w)<1$ and hence
\begin{align}\label{eq.A_4.2.7}
 L(w)+Q_2 =L(w)\left\{1+\frac{Q_2}{L(w)}\right\}\leq \frac{L(w)}{1-Q_2/L(w)}.
\end{align}
Then, by \eqref{eq.A_4.2}, \eqref{eq.A_4.2.6} and \eqref{eq.A_4.2.7}, we have
\begin{align}\label{eq.A_4.4}
\frac{\{\int_0^\infty\{1/L(g)\}F( w ,g)\pi(g)\rd g\}^2}{\int_0^\infty F( w ,g)\pi(g)\rd g}
\geq  \frac{2^{d/2-1}\Gamma(d/2-1)}{w^{d/2-1}L(w)}\left(1-\frac{Q_1+Q_2}{L(w)}\right).
\end{align}
Let $w_2=\max(w_1,\exp(Q_2))$ and $Q_3=Q_1+Q_2$. 
Then, by \eqref{eq.A_4.4}, we have
\begin{align}
& \frac{1}{C_d}\left(\int_{\|x\|^2\leq w_2}+\int_{\|x\|^2> w_2}\right)
\frac{1}{\|x\|^2}\frac{\{\int_0^\infty\{1/L(g)\}F( \|x\|^2 ,g)\pi(g)\rd g\}^2}{\int_0^\infty F( \|x\|^2 ,g)\pi(g)\rd g}\rd x  \\
&\geq 
\frac{1}{C_d}\int_{\|x\|^2> w_2} \frac{1}{\|x\|^2}\frac{(\|x\|^2)^{-d/2+1}2^{d/2-1}\Gamma(d/2-1)}{L(\|x\|^2)}\left(1-\frac{Q_3}{L(\|x\|^2)}\right)\rd x \\
&=
\int_{w_2}^\infty \frac{\rd g}{gL(g)} -\int_{w_2}^\infty \frac{Q_3\rd g}{g\{L(g)\}^2}. \label{eq.A_4.5}
\end{align}
By \eqref{eq.A_4.1}, \eqref{eq.A_4.5} and Part \ref{lem:gamma.2} of Lemma \ref{lem:gamma}, we have
\begin{align*}
&  \frac{1}{C_d} \int_{\mathbb{R}^d} \frac{m_i(\|x\|^2)}{\|x\|^2}A^2_4(\|x\|^2)\rd x \\
&\leq \int_0^\infty \frac{\rd g}{(g+1)L(g)}
-\int_{w_2}^\infty \frac{\rd g}{gL(g)} +\int_{w_2}^\infty \frac{Q_3\rd g}{g\{L(g)\}^2} \\
&= \int_0^{w_2} \frac{\rd g}{(g+1)L(g)}
-\int_{w_2}^\infty \frac{\rd g}{g(g+1)L(g)} +\int_{w_2}^\infty \frac{Q_3\rd g}{g\{L(g)\}^2} \\
&= \int_0^{w_2} \frac{\rd g}{(g+1)L(g)}
+\int_{w_2}^\infty \frac{Q_3\rd g}{g\{L(g)\}^2} <\infty.
\end{align*}

\section{Lemmas}
\begin{lem}\label{lem:bound.risk}
\begin{align}
\sup_{x} \|\nabla_x\log m_\pi(\|x\|^2;a,b,c)\|^2<\infty.
\end{align}
\end{lem}
\begin{proof}
By \eqref{estimator_pi}, let
\begin{align}\label{lem.bounded.0}
f(\|x\|^2)=\|\nabla_x\log m_\pi(\|x\|^2;a,b,c)\|=
\|x\| 
\frac{\int_0^\infty(g+1)^{-d/2-1}\exp\left(-\frac{\|x\|^2}{2(g+1)}\right)\pi(g;a,b,c)\rd g}
{\int_0^\infty(g+1)^{-d/2}\exp\left(-\frac{\|x\|^2}{2(g+1)}\right)\pi(g;a,b,c)\rd g},
\end{align}
where
\begin{align*}
 m_\pi(\|x\|^2;a,b,c)=\int_0^\infty(g+1)^{-d/2}\exp\left(-\frac{\|x\|^2}{2(g+1)}\right)\pi(g;a,b,c)\rd g.
\end{align*}
Clearly $ f(0)=0$ since
\begin{align*}
 f(0)=0\times
\frac{\int_0^\infty(g+1)^{-d/2-1}\pi(g;a,b,c)\rd g}
{\int_0^\infty(g+1)^{-d/2}\pi(g;a,b,c)\rd g}=0.
\end{align*}
Note,  as in \eqref{tauberian}, 
\begin{align}
\lim_{t\to\infty} \frac{t^{d/2-1}m_{\pi}(t;a,b,c)}{\pi(t;a,b,c)}=\Gamma(d/2-1-a)2^{d/2-1-a}.\label{lem.bounded.1}
\end{align}
Similarly,
\begin{align}
\lim_{t\to\infty} \frac{t^{d/2}\int_0^\infty(g+1)^{-d/2-1}\exp\left(-\frac{\|x\|^2}{2(g+1)}\right)\pi(g;a,b,c)\rd g}{\pi(t;a,b,c)}=\Gamma(d/2-a)2^{d/2-a}.\label{lem.bounded.2}
\end{align}
Hence, by \eqref{lem.bounded.0}, \eqref{lem.bounded.1} and \eqref{lem.bounded.2}, we have
\begin{align*}
 \lim_{t\to\infty}t^{1/2}f(t)=d/2-1-a
\end{align*}
which implies that $ \lim_{t\to\infty}f(t)=0$. Together with $f(0)=0$, $f(t)$ is bounded.
\end{proof}

\begin{lem}\label{lem:integrability.1}
\begin{enumerate}
 \item\label{lem:integrability.1.1} For either $a>0$ or \{$a=0$ and $c>1$\},
\begin{align}
  \int_1^\infty\frac{\rd g}{g\pi(g;a,b,c)} <\infty.
\end{align}
\item\label{lem:integrability.1.2} For either $a<0$ or \{$a=0$ and $c<-1$\},
\begin{align}
  \int_0^\infty\frac{\pi(g;a,b,c)}{g+1}\rd g <\infty.
\end{align}
\item\label{lem:integrability.1.3} For either $a<0$ or \{$a=0$ and $c\leq 1$\},
\begin{align}
\int_1^\infty\frac{\rd g}{g\pi(g;a,b,c)} =\infty.
\end{align}
\end{enumerate}
\end{lem}

\begin{proof}
\ [Part \ref{lem:integrability.1.1}] 
Let
\begin{align*}
 c_*=
\begin{cases}
 2 & a>0 \\
c&a=0
\end{cases}
\end{align*}
and
\begin{align}
 f_1(g)=\frac{\{\log (g+1) +1\}^{c_*-c}}{(g+1)^{a}}\left(\frac{g+1}{g}\right)^{b+1},
\end{align}
which is bounded for $g\in(1,\infty)$, for either $a>0$ or \{$a=0$ and $c>1$\}, since $f_1(1)<\infty$ and $f_1(\infty)<\infty$.
Then
\begin{align}
 \int_1^\infty\frac{\rd g}{g\pi(g;a,b,c)} 
&=\int_1^\infty \frac{f_1(g)\rd g}{(g+1)\left\{\log(g+1)+1\right\}^{c_*}}\\
&\leq\max_{g\geq 1}f_1(g)\int_1^\infty \frac{\rd g}{(g+1)\left\{\log(g+1)+1\right\}^{c_*}}\\
&=\max_{g\geq 1}f_1(g)\frac{(\log 2+1)^{1-c_*}}{c_*-1},
\end{align}
which completes the proof of Part \ref{lem:integrability.1.1}.

[Part \ref{lem:integrability.1.2}]
For $0<g<1$,
\begin{align}\label{small.g}
 \frac{\pi(g;a,b,c)}{g+1}=\frac{g^b\left\{\log(g+1)+1\right\}^{c}}{(g+1)^{-a+(b+1)}}\leq 
\max\{1,(\log 2+1)^c\}g^b.
\end{align}
For $g\geq 1$, let
\begin{align*}
 c_*=
\begin{cases}
 2 & a<0 \\
-c & a=0
\end{cases}
\end{align*}
and
\begin{align}
 f_2(g)=(g+1)^{a}\left(\frac{g}{g+1}\right)^{b}\{\log (g+1) +1\}^{c+c_*}.
\end{align}
which is bounded for $g\in(1,\infty)$ since $f_1(1)<\infty$ and $f_1(\infty)<\infty$.
Then, for $g\geq 1$, we have
\begin{align}\label{large.g}
 \frac{\pi(g;a,b,c)}{g+1}=\frac{f_2(g)}{(g+1)\{\log(g+1)+1\}^{c_*}}\leq \frac{\max_{g\geq 1}f_2(g)}{(g+1)\{\log(g+1)+1\}^{c_*}}.
\end{align}
By \eqref{small.g} and \eqref{large.g}, we have
\begin{align*}
 \int_0^\infty\frac{\pi(g;a,b,c)}{g+1}\rd g
&\leq \max\{1,(\log 2+1)^c\}\int_0^1 g^b\rd g+\int_1^\infty
\frac{\max_{g\geq 1}f_2(g)}{(g+1)\{\log(g+1)+1\}^{c_*}}\rd g \\
&=\frac{\max\{1,(\log 2+1)^c\}}{b+1}+\frac{(\log 2+1)^{1-c_*}}{c_*-1},
\end{align*}
which completes the proof of Part \ref{lem:integrability.1.2}.

[Part \ref{lem:integrability.1.3}] 
Let
\begin{align}
 f_3(g)=\frac{\{\log (g+1) +1\}^{1-c}}{(g+1)^{a}}\left(\frac{g+1}{g}\right)^{b+1}.
\end{align}
which is positive and bounded away from $0$, for either $a<0$ or \{$a=0$ and $c\leq 1$\}.
Then
\begin{align}
 \int_1^\infty\frac{\rd g}{g\pi(g;a,b,c)} 
&=\int_1^\infty \frac{f_3(g)\rd g}{(g+1)\left\{\log(g+1)+1\right\}}\\
&\geq\min_{g\geq 1}f_3(g)\int_1^\infty \frac{\rd g}{(g+1)\left\{\log(g+1)+1\right\}}\\
&=\min_{g\geq 1}f_3(g)\left[\log(\{\log(g+1)+1\})\right]_1^\infty \\
&=\infty,
\end{align}
which completes the proof of Part \ref{lem:integrability.1.3}.

\end{proof}

\begin{lem}\label{lem:gamma}
Assume $d\geq 3$. 
\begin{enumerate}
 \item \label{lem:gamma.1}For $\alpha>0$, 
\begin{equation}
 \int_{\mathbb{R}^d}\frac{1}{\|x\|^2}\exp\left(-\frac{\|x\|^2}{\alpha}\right)\rd x =\frac{2\pi^{d/2}\alpha^{d/2-1}}{d-2}.
\end{equation}
\item \label{lem:gamma.2}For $g\geq 0$, let
\begin{align}
 F( w ,g)=(g+1)^{-d/2}\exp\left(-\frac{ w }{2(g+1)}\right).
\end{align}
Then
\begin{align}
\int_{\mathbb{R}^d}\frac{F(\|x\|^2,g)}{\|x\|^2}\rd x=\frac{C_d}{g+1}, \ \text{ where } \ C_d=\frac{\pi^{d/2}2^{d/2}}{d-2}.
\end{align}
\end{enumerate}
\end{lem}

\begin{proof}
For Part \ref{lem:gamma.1}, we have
\begin{equation}\label{identity}
 \begin{split}
 \int_{\mathbb{R}^d}\frac{1}{\|x\|^2}\exp\left(-\frac{\|x\|^2}{\alpha}\right)\rd x 
&=\frac{\pi^{d/2}}{\Gamma(d/2)}\int_0^\infty t^{d/2-1-1}\exp(-t/\alpha)\rd t \\
&=\frac{\pi^{d/2}}{\Gamma(d/2)}\Gamma(d/2-1)\alpha^{d/2-1}\\
&=\frac{2\pi^{d/2}\alpha^{d/2-1}}{d-2}.
\end{split}
\end{equation}
Part \ref{lem:gamma.2} follows from Part \ref{lem:gamma.1}.
\end{proof}

\end{document}